\documentclass[11pt]{amsart}

\usepackage{color}
\usepackage{amscd,amssymb,amsmath,graphicx,verbatim}
\usepackage{wasysym}
\usepackage{hyperref}
\usepackage{lscape} 
\usepackage{fullpage} 
\usepackage{tikz}
\usetikzlibrary{shapes,arrows}
\usepackage[all]{xy}
\newtheorem{theorem}{Theorem}[section]
\newtheorem{lemma}[theorem]{Lemma}
\newtheorem{corollary}[theorem]{Corollary}
\newtheorem{proposition}[theorem]{Proposition}

\theoremstyle{definition}
\newtheorem{definition}[theorem]{Definition}

\newtheorem{example}[theorem]{Example}
\newtheorem{examples}[theorem]{Examples}

\theoremstyle{remark}
\newtheorem{remark}[theorem]{Remark}
 \DeclareMathOperator{\Hom}{Hom}
\DeclareMathOperator{\Ext}{Ext} \DeclareMathOperator{\End}{End}

\DeclareMathOperator{\Tor}{Tor} 

\newcommand{\Dcal}{\ensuremath{\mathcal{D}}}
\newcommand{\Xcal}{\ensuremath{\mathcal{X}}}

\newcommand{\Tcal}{\ensuremath{\mathcal{T}}}
\newcommand{\Gcal}{\ensuremath{\mathcal{G}}}
\newcommand{\Fcal}{\ensuremath{\mathcal{F}}}
\newcommand{\Ccal}{\ensuremath{\mathcal{C}}}

\newcommand{\Nbb}{\ensuremath{\mathbb{N}}}

\newcommand{\Acal}{\ensuremath{\mathcal{A}}}
\newcommand{\Wcal}{\ensuremath{\mathcal{W}}}

\newcommand{\Scal}{\ensuremath{\mathcal{S}}}

\newcommand{\Ucal}{\ensuremath{\mathcal{U}}}
\newcommand{\Vcal}{\ensuremath{\mathcal{V}}}

\newcommand{\D}{{\rm D}}
\newcommand{\Kbb}{\mathbb{K}}
\newcommand{\Z}{\mathbb{Z}}

\newcommand{\ra}{\rightarrow}
\DeclareMathOperator{\Ker}{Ker} 
\DeclareMathOperator{\Coker}{Coker}

\numberwithin{equation}{section}

\newcommand{\Spec}[1]{\mbox{\rm{Spec}}(#1)}

\newcommand{\Add}{\mbox{\rm{Add\,}}}
\newcommand{\Prod}{\mbox{\rm{Prod\,}}}

\newcommand{\Gen}{\mbox{\rm{Gen\,}}}
\newcommand{\gen}{\mbox{\rm{gen\,}}}

\newcommand{\Cogen}{\mbox{\rm{Cogen\,}}}

\newcommand{\im}{\mbox{\rm{Im\,}}}

\newcommand{\rfmod}[1]{\mbox{\rm{mod}-}{#1}}
\newcommand{\rmod}[1]{\mbox{\rm{Mod}-}{#1}}
\newcommand{\lfmod}[1]{{#1}\mbox{-\rm{mod}}}
\newcommand{\lmod}[1]{{#1}\mbox{-\rm{Mod}}}
\newcommand{\ModA}{\ensuremath\mbox{\rm{Mod}-$A$}}
\newcommand{\modA}{\ensuremath\mbox{\rm{mod}-$A$}}
\newcommand{\ProjA}{\ensuremath\mbox{\rm{Proj}-$A$}}\newcommand{\projA}{\ensuremath\mbox{\rm{proj}-$A$}}
\newcommand{\Aproj}{\ensuremath\mbox{A-\rm{proj}}}

\newcommand{\ModB}{\ensuremath\mbox{\rm{Mod}-$B$}}

\newcommand{\Ann}[1]{\mbox{\rm{Ann}}(#1)}

\newcommand{\Ab}{\mathrm{Ab}}

\newcommand{\C}{\mathcal{C}}

\newcommand{\Mor}[1]{\mbox{\rm{Mor}}(#1)}

\begin{document}
\title{On the abundance of silting modules}
\author{Lidia Angeleri H\"ugel}
\address{Lidia Angeleri H\"ugel, Dipartimento di Informatica - Settore di Matematica, Universit\`a degli Studi di Verona, Strada le Grazie 15 - Ca' Vignal, I-37134 Verona, Italy} \email{lidia.angeleri@univr.it}

\maketitle

\begin{abstract} Silting modules are abundant. Indeed, they parametrise the definable torsion classes over a noetherian ring, and the hereditary torsion pairs of finite type over a commutative ring. Also the universal localisations of a hereditary ring, or of a finite dimensional algebra of finite representation type, can be  parametrised by silting modules. In these notes, we give a brief introduction to the fairly recent concepts of  silting and cosilting module, and we explain the classification results mentioned above.
\end{abstract}
 
\bigskip

\bigskip

\noindent
\section{Introduction}
The notion of a (compact) silting complex was introduced by Keller and Vossieck \cite{KV}, and it was  later rediscovered in work of 
 Aihara-Iyama, Keller-Nicol\`as, K\"onig-Yang,  Mendoza-S\'aenz-Santiago-Souto Salorio, and Wei \cite{AI,KN,MSSS,Wei,KY}. This renewed interest was motivated by cluster theory, and also by the interplay with certain torsion pairs in triangulated  categories. In fact, silting complexes are closely related with  t-structures and co-t-structures in the derived category. This connection was recently extended to non-compact silting complexes, and  more generally, to silting objects in triangulated categories \cite{AMV1,PV,NSZ}.
 
 \smallskip

 But silting is also closely related to torsion pairs and localisation of abelian categories, as well as to ring theoretic localisation. The aim of these notes is to review  some of these connections. We will see that localisation techniques provide  constructions of silting objects and lead to classification results over several important classes of rings. There is also an interesting interaction with  combinatorial  aspects of silting. Certain posets studied in cluster theory have a ring theoretic interpretation which sheds new light on their structure. On the other hand, silting theory yields a new approach to general questions  on homological properties of ring epimorphisms. 
 
  \smallskip
 
We will consider silting modules, that is, the modules that arise as zero cohomologies of (not necessarily compact) 2-term silting complexes over an arbitrary ring. These modules were introduced in \cite{AMV1}. They provide a generalisation  of  (not necessarily finitely generated) tilting modules. Moreover, over a  finite dimensional algebra, the finitely generated silting modules are precisely the support $\tau$-tilting modules introduced in \cite{AIR} 
and studied in cluster theory.

\smallskip

In representation theory, one usually studies finite dimensional tilting or support $\tau$-tilting modules up to isomorphism and multiplicities. Similarly, in the infinite dimensional case, it is convenient to study silting classes rather than modules. The silting class $\Gen T$ given by a silting module $T$ consists of all $T$-generated modules, and it determines the additive closure $\Add T$ of $T$. It has several useful closure properties, making it  a definable torsion class. A silting class thus provides every module $M$ with a (minimal) right $\Gen T$-approximation given by the trace of $T$ in $M$, and with a left $\Gen T$-approximation. Furthermore,  $\Gen T$ - even when $T$ is not finitely generated - satisfies an important finiteness condition: it is determined by a set $\Sigma$ of morphisms between finitely generated projective modules \cite{MS2}. This ``finite type'' result extends the analogous result for tilting modules from \cite{BH} stating that every tilting class is determined by a set of finitely presented modules of projective dimension at most one.

\smallskip

It turns out that in many respects the dual concept  of a 
 cosilting module studied in \cite{BP,WZ,BZ} is more accessible. 
 Cosilting modules are pure-injective, and cosilting classes are definable torsionfree classes, in fact, they coincide with the definable torsionfree classes. Notice that the dual result is not true in general: not every definable torsion class is silting. Indeed, there are more cosilting than silting classes, because in general cosilting modules are not of ``cofinite type'', that is, they need not be determined by a set $\Sigma$ as above. Also this phenomenon was already known for cotilting modules \cite{Ba}, a cosilting example is exhibited in Example~\ref{notcoft}.
 
 \smallskip
 
 However, over a noetherian ring,  every torsion pair with definable torsionfree class is generated by a set of finitely presented modules. This allows to show that all cosilting modules are of cofinite type. Then we obtain dually that the silting classes coincide with the  definable torsion classes \cite{AH}. This can be regarded as a ``large'' analog of   \cite[Theorem 2.7]{AIR} stating that over a finite dimensional algebra $A$, there is a bijection between isomorphism classes of basic support $\tau$-tilting modules and functorially finite torsion classes  in $\modA$.

\smallskip

Similarly, one can show that over a commutative ring, silting modules correspond bijectively to hereditary torsion pairs of finite type, and if the ring is commutative noetherian, they are parametrised by subsets of the spectrum closed under specialisation \cite{AH}.

\smallskip

Silting is thus often related to Gabriel localisation of  module categories. But it is also related to ring theoretic localisation, in particular to  universal localisation of rings in the sense of \cite{Sch}. In fact, by \cite{AMV2} every partial silting module over a ring $A$ induces a ring epimorphism $A\to B$ with nice homological properties, and it is proved in \cite{MS2} that every universal localisation of $A$ arises in this way from some partial silting module.  Further, in certain cases, the universal localisations of $A$, or the homological ring epimorphisms starting in $A$, are parametrised by silting modules. 

\smallskip 

More precisely, one obtains a parametrisation by minimal silting modules, that is, silting modules satisfying a condition which ensures the existence of a minimal left $\Gen T$-approximation ${f:A\to T_0}$ for the regular module $A$. For example,  
when $A$ is a finite dimensional algebra, every finite dimensional silting module is minimal. Moreover,   the cokernel of the approximation $f$ is a partial silting module uniquely determined by $T$, so one can associate to $T$  a ring epimorphism  $A\to B$, which turns out to be a universal localisation of $A$. If $A$ is of  finite representation type, or more generally, $\tau$-tilting finite in the sense of \cite{DIJ}, this assignment yields a bijection between silting modules and universal localisations  of $A$, see \cite{MS1}.

\smallskip

A similar result holds true when $A$ is a  hereditary ring. As shown in \cite{AMV2}, the minimal silting $A$-modules are in bijection with the universal localisations  of $A$, which by \cite{KSt} coincide  with the homological ring epimorphisms. If $A$ is a finite dimensional hereditary algebra, this correspondence yields a bijection between finite dimensional support tilting modules and functorially finite wide subcategories of $\modA$,  recovering results from \cite{IT, Ma}. The combinatorial interpretation of finite dimensional support tilting modules in terms of noncrossing partitions or clusters can then be translated into ring theoretic terms by considering the poset of all universal localisations of $A$. Notice that universal localisations of $A$  can have infinite dimension. Leaving the finite dimensional world, however, has the advantage that the poset becomes a lattice, while the same is not true if we restrict to ring epimorphisms with finite dimensional target, see e.g.~Example~\ref{posets}(1).

\smallskip

The paper is organised as follows. Section 2 is devoted to the notion of a silting module. In Section 3, we  collect some results on cosilting modules scattered in the literature. The interplay between silting and cosilting modules is discussed in Section 4, which also contains the classification results over noetherian or commutative rings mentioned above. In Section 5, we review the connections with ring epimorphisms, and we explain the parametrisation of all universal localisations of a hereditary ring. Finally, in Section 6, we focus on finite dimensional algebras and on the poset of their universal localisations.

\bigskip

{\bf Acknowledgements.} These notes arose from my talk at the ``Maurice Auslander Distinguished Lectures and Conference 2015'' in Woods Hole, and from  lectures delivered  at  the University of Verona. I would like to thank the organisers  for the invitation to a very nice and interesting conference. Furthermore, I thank 
Simion Breaz, 
Frederik Marks, Jan  \v{S}\v{t}ov\'\i\v{c}ek, and Jorge Vit\'oria for many valuable comments on a first version of the manuscript, and I thank  Pavel P\v{r}\'ihoda for providing the argument in Example~\ref{exs2}(3).

\bigskip

{\bf Notation.}
Throughout the paper, $A$ will denote a ring, $\ModA$ the category of all right $A$-modules,  and  $\rfmod{A}$ the category of all finitely presented right $A$-modules. The corresponding categories of left $A$-modules are denoted by $\lmod{A}$ and $\lfmod{A}$. The unbounded derived category of $\ModA$ is denoted by $\D(A)$.
We further denote by $\ProjA$ and $\projA$ the full subcategory of $\ModA$ consisting of all projective and all finitely generated projective right $A$-modules, respectively, and we write $\;^\ast =\Hom_A(-,A)$.

 Given a subcategory $\mathcal{C}$ of $\ModA$,  we denote   $$\mathcal{C}^o=\{M \in \ModA \mid \Hom_A({\Ccal},{M})=0\},$$ $$\mathcal{C}^{\perp_1}=\{M \in \ModA \mid \Ext^{1}_A({\Ccal},{M})=0\}.$$ The classes ${}^o\Ccal$ and ${}^{\perp_1}\Ccal$ are defined dually.
Moreover, $\Add\Ccal$ and $\Prod\Ccal$ are the subcategories of $\ModA$  formed by the modules that are isomorphic to a direct summand of a coproduct  of modules in $\Ccal$, or of a product of such modules, respectively.
 $\Gen\Ccal$ is the subcategory of the $\Ccal$-generated modules, i.~e.~the epimorphic images of modules in $\Add\Ccal$, and $\Cogen\Ccal$ is defined dually. When $\Ccal$ just consists of a single module $M$, we write $\Add M, \Prod M, \Gen M, \Cogen M$.

Furthermore, we denote by $\Mor{\mathcal{C}}$ the class of all morphisms in $\ModA$ between objects in $\mathcal{C}$.
Given a set of morphisms $\Sigma\subset\Mor\Ccal$, we consider the classes
\[\begin{array}{lcl}
\Dcal_\Sigma&=&\{X\in \ModA \,\mid\, \Hom_A(\sigma,X)\ \text{is surjective for all } \sigma\in\Sigma \};\\
\Xcal_\Sigma&=&\{X\in \ModA \,\mid\, \Hom_A(\sigma,X)\ \text{is bijective for all } \sigma\in\Sigma \};\\
\mathcal{C}_\Sigma&=&\{X\in \rmod A \,\mid\, \Hom_{A}({X},{\sigma})\ \text{is surjective  for all } \sigma\in\Sigma\};\\ 
\mathcal{F}_\Sigma&=& \{X \in \lmod{A} \,\mid\, \sigma \otimes_A X \text{ is injective for all } \sigma\in\Sigma\}.\end{array}\]
Again, when $\Sigma=\{\sigma\}$, we  write $\Dcal_\sigma, \Xcal_\sigma, \Ccal_\sigma, \Fcal_\sigma$.

\bigskip

\bigskip
\section{Silting modules}

We start out by briefly reviewing the notion of a (not necessarily compact) silting complex. \begin{definition}  \cite{Wei}
A bounded complex of projective $A$-modules $\sigma$ is said to be {\em silting} if 
\begin{enumerate}
\item $\Hom_{\D(A)}(\sigma,\sigma^{(I)}[i])=0$, for all sets $I$ and $i>0$.
\item[(2)] the smallest triangulated subcategory of $\D(A)$ containing $\Add\sigma$ is ${\rm K^b}(\ProjA)$.
\end{enumerate}
\end{definition}

A prominent role is played by the silting complexes of length two which were studied in \cite{AIR} in connection with cluster mutation. Their endomorphism ring has interesting properties investigated in \cite{BuanZhou1,BuanZhou2, BuanZhou3}.
Another important feature is the fact that 
  2-term silting complexes (which we  always assume concentrated in cohomological degrees -1 and 0) are determined by their zero cohomology. In fact, identifying a complex $\ldots0\to P_{-1}\stackrel{\sigma}{\longrightarrow} P_0\to 0\ldots$ with the morphism $\sigma$ in $\Mor\ProjA$, one obtains the following result which goes back to work of Hoshino-Kato-Miyachi \cite{HKM}.
 
\begin{proposition} Let $\sigma$ be 2-term complex in ${\rm K^b}(\ProjA)$ and $T=H^0(\sigma)$.
 Then $\sigma$ is a silting complex if and only if the class
$\Gen T$ of $T$-generated modules coincides with  
$\Dcal_\sigma$.
\end{proposition}

We are interested in the modules that occur in this way. 
\begin{definition}
A right $A$-module $T$ is
{\em silting} if it admits a projective presentation $ P_{-1}\stackrel{\sigma}{\longrightarrow} P_0\to T\to 0$ such that 
 $\Gen T=\Dcal_\sigma.$ The class $\Gen T$ is then called a {\em silting class}.
 
 \smallskip
 
 We say that two silting modules $T,T'$ are {\em equivalent} if they generate the same silting class. By \cite[Section 3]{AMV1}, this is equivalent to $\Add T=\Add T'$.
\end{definition}

\begin{examples}\label{exs}
(1) A module $T$ is silting with respect to a monomorphic projective presentation $\sigma:P_{-1}{\hookrightarrow} P_0$ if and only if $\Gen T=\Ker\Ext_A^1(T,-),$ i.~e.~$T$ is a {\it tilting} module (of projective dimension at most one, not necessarily finitely generated).

\smallskip

(2) If $A$ is a finite dimensional algebra over a field, and $T\in \rfmod{A}$, then  $T$ is silting if and only if it is {\it support $\tau$-tilting} in the sense of \cite{AIR}.

\smallskip

(3) Let $A$ be the path algebra of   the quiver $Q$ having two vertices, 1 and 2, and countably many arrows from 1 to 2. Let $P_i=e_iA$ be the indecomposable projective $A$-module  for $i=1,2$. Then $T:=S_2$ with the
projective presentation
$$\xymatrix{0\ar[r] & P_1^{(\Nbb)}\ar[r]^\sigma & P_2\ar[r] & T\ar[r] & 0,}$$
 is a silting module (of projective dimension one, not  finitely presented) which is not tilting. 
 
Indeed,  it is not tilting as the class $\Gen T$ consists precisely of the semisimple injective $A$-modules, so $\Gen T=\Ker\Hom_A(P_1,-)\subsetneq \Ker\Ext_A^1(T,-)$. But $T$ is  silting with respect to the projective presentation $\gamma$  of $T$ obtained as the direct sum of $\sigma$ with the trivial map $P_1\ra 0$, since $\mathcal{D}_{\gamma}=\Ker\Ext_A^1(T,-)\cap \Ker\Hom_A(P_1,-)=\Ker\Hom_A(P_1,-)=\Gen T.$
\end{examples}

Notice  that every class of the form $\Dcal_\sigma$ for a morphism $\sigma\in\Mor\ProjA$ 
is closed under extensions, epimorphic images, and direct products. In the silting case, $\Dcal_\sigma=\Gen T$ is also closed under coproducts,  and it is therefore a torsion class. 
Moreover, even though $\sigma$ need not belong to $\Mor\projA$, the silting class $\Dcal_\sigma$  is determined by a set $\Sigma$ of morphisms in $\Mor\projA$, that is,  $\Dcal_\sigma=\Dcal_\Sigma$. This property is obtained in \cite{MS2} as a consequence of  the analogous result  for tilting modules proved in \cite{BH}; an alternate proof  in \cite{AH} uses that $\Dcal_\sigma$ is definable
(as shown in  \cite[3.5 and 3.10]{AMV1}).
Recall that a class of modules is said to be  {\em definable} if it is closed under direct limits, direct products, and pure submodules.

 \smallskip
 
In fact, these  properties characterise silting classes.

\begin{theorem}\cite{MS2}\label{deft}
The following statements are equivalent for a full subcategory $\Dcal$ of $\ModA$. 
		\begin{enumerate}
			\item[(i)] $\Dcal$ is a silting class;
			\item[(ii)]  there is	$\sigma \in \Mor{\ProjA}$ such that  $\Dcal=\Dcal_\sigma$, and $\Dcal$ is  closed for coproducts;	
			\item[(iii)] there is	a set $\Sigma\subset\Mor\projA$ such that  $\Dcal=\Dcal_\Sigma$.		\end{enumerate}
			Moreover, $\Dcal$ is a tilting class if and only if $\Dcal=\Dcal_\Sigma$ for a set  $\Sigma$ of monomorphisms in $\Mor\projA$.
\end{theorem}

A silting class $\Dcal$, being a  torsion class,   provides any module $M$ with a minimal right $\Dcal$-approximation $d(M)\hookrightarrow M$ where $d$ denotes the torsion radical corresponding to $\Dcal$. Moreover, being definable, $\Dcal$ also provides $M$ with a left $\Dcal$-approximation. 
In fact, the existence of special left approximations can be used to detect the torsion classes which are generated by a tilting module.
\begin{theorem} \cite{ATT}\label{attilting}
A torsion class $\Tcal$ in $\ModA$  is of the form $\Tcal=\Gen T$ for some tilting module $T$ if and only if for every $A$-module $M$ (or equivalently, for $M=A$) there is a short exact sequence $0\to M\to B\to C\to 0$ with $B\in\Tcal$ and $C\in {}^{\perp_1} \Tcal$.
\end{theorem}
 At the end of the next section, we will discuss how this result extends to silting modules.

\bigskip

\section{Cosilting modules}

When dealing with classification results for silting modules, it turns out that  the dual concept of a cosilting module, which was
introduced by Breaz and Pop in \cite{BP}, is often more accessible. In order to study the interplay between the two notions, it will be convenient to consider right silting and left cosilting modules. 

\begin{definition} A left $A$-module $C$ is {\it cosilting} if it admits an injective copresentation  $0\to C\longrightarrow E_{0}\stackrel{\omega}{\longrightarrow} E_1$ such that the class $\Cogen C$ of $C$-cogenerated modules coincides with the class 
$\mathcal{C}_\omega.$ The class $\Cogen C$ is then called a {\it cosilting class}.

\smallskip

Two cosilting modules $C,C'$ are said to be {\em equivalent} if they cogenerate the same class $\Cogen C=\Cogen C'$. It will follow from Remark~\ref{prod} below that this is equivalent to  $\Prod C=\Prod C'$. 
\end{definition}

Dually to Example~\ref{exs}(1), we have that $C$ is cosilting with respect to an epimorphic injective copresentation  $E_{0}\stackrel{\omega}{\twoheadrightarrow} E_1$  if and only if $\Cogen C=\Ker\Ext^1_A(-,C),$ i.~e.~$C$ is a {\it cotilting} module (of injective dimension at most one, not necessarily finitely generated).

\smallskip

Further,   dually to the silting case, we see that  every cosilting class $\Ccal_\omega=\Cogen C$ is a torsionfree class.
How to describe the torsionfree classes that arise in this way?
For  cotilting modules there is the following characterisation in terms of the existence of approximations.

\begin{theorem} \cite{ATT}\label{att}
A torsionfree class $\Fcal$ in $\lmod{A}$ is of the form $\Fcal=\Cogen C$ for some cotilting module $C$ if and only if for every left $A$-module $M$ (or equivalently, for an injective cogenerator  $E$ of  $\lmod{A}$) there is a short exact sequence $0\to  K\to L\to M\to 0$ with $L\in\Fcal$ and $K\in \Fcal^{\perp_1} $.
\end{theorem}

An extension 
of this theorem to cosilting modules was first obtained by Zhang and Wei \cite{WZ}, as a result of their comparison of several  notions generalizing the definition of a cotilting  module. Here we focus only on  the arguments relevant to the notion of a cosilting module, which we collect   below for the reader's convenience. 
We start with the following  observations.

\begin{lemma}\label{corigid}
Let $C$ be a left $A$-module with  injective copresentation  $0\to C\longrightarrow E_{0}\stackrel{\omega}{\longrightarrow} E_1$. 
\begin{enumerate}
\item A module $X$ belongs to the class $\C_\omega$ if and only if $\Hom_{\D(A)}(X,\omega[1])=0$, if and only if $\Hom_{\D(A)}(\sigma,\omega[1])=0$ for any injective copresentation  $0\to X\longrightarrow I_{0}\stackrel{\sigma}{\longrightarrow} I_1$.
\item Assume that  $0\to C\longrightarrow E_{0}\stackrel{\omega}{\longrightarrow} E_1$ is a minimal injective copresentation. Then 
$\Cogen C\subset {}^{\perp_1} C$ if and only if $\Hom_{\D(A)}(\omega^I,\omega[1])=0$ for any set $I$, if and only if $\Hom_{\D(A)}(X,\omega[1])=0$ for all $X\in\Cogen C$.
\end{enumerate}
\end{lemma}
\begin{proof} 
(1) is shown by standard arguments.
Moreover, using (1), the statement in (2) can be rephrased as follows: $\Cogen C\subset {}^{\perp_1} C$ if and only if all products of copies of $C$ are in $\Ccal_\omega$, if and only if $\Cogen C\subset\Ccal_\omega$. Now, since
  $\Ccal_\omega$ is always closed under submodules, and  $\Ccal_\omega\subset {}^{\perp_1} C$, the second condition means $\Cogen C\subset\Ccal_\omega$, which in turn  entails the first condition  $\Cogen C\subset {}^{\perp_1} C$.
 So it remains to show that  the first condition implies the second. 
We sketch an argument from  \cite[Lemma 4.13]{WZ}. Consider a cardinal $\kappa$ and a map $f: C^\kappa\to E_1$. The map $\omega$ factors as  $\omega=e\circ \omega'$ with  $e:\im\omega\hookrightarrow E_1$ being an injective envelope. Setting $Z=\Coker\omega$, we obtain  the following commutative diagram with exact rows
\begin{equation}
\xymatrix{&0\ar[r]&K\ar[d]^{h}\ar[r]^{i} & C^\kappa\ar[d]^{f}\ar[r]^{gf} & Z\ar @{=}[d]&\\
0\ar[r] & C\ar[r]^{} &E_0\ar[r]^{\omega} &E_1\ar[r]^{g} & Z\ar[r]&0}
\end{equation}
where $h$ is constructed by first taking the map $f':K\to \im\omega$ induced by $f$, and then lifting $f'$ to $h$ thanks to the fact that 
$K\in\Cogen C\subset {}^{\perp_1} C$. Now the injectivity of $E_0, E_1$ yields maps $s_0: C^\kappa\to E_0$ and $s_1:Z\to E_1$ such that  $f=\omega s_0+ s_1(gf)$, and one verifies $\im (f-\omega s_0)\cap\im \omega=0$. Since $\im \omega$ is an essential submodule of $E_1$, it follows  $f=\omega s_0$, as required. 
\end{proof}

\begin{remark}\label{prod}
 A module $M$ in a torsionfree class $\Fcal$ is said to be {\it Ext-injective} in $\Fcal$ if $\Ext^1_A(\Fcal, M)=0$.
Notice that the condition $\Cogen C\subset {}^{\perp_1} C$ implies that $(^oC,\Cogen C)$ is a torsion pair, and it can be rephrased by saying that $C$ is an  {Ext-injective} object in the torsionfree class $\Cogen C$.

Now assume that $C$ is a cosilting module with respect to an injective copresentation $\omega$. Since $\Ccal_\omega\subset {}^{\perp_1} C$, it follows that $C$ is Ext-injective in $\Cogen C$. More precisely, $\Prod C$ is the class of all Ext-injective modules in $\Cogen C$. Indeed, 
for every module  $M\in\Cogen C$ there is an embedding $f:M\to C'$ with $C'\in\Prod C$ and $\Coker f\in\Cogen C$. If $M$ is {Ext-injective} in $\Cogen C$, then $f$ is a split monomomorphism and $M\in\Prod C$.  For details we refer to \cite[Lemma 2.3 and 3.3, Proposition 3.10]{AMV1} where the dual statements are proved.
\end{remark}

\smallskip

The following result is inspired by the proof of  \cite[Proposition 4.13]{WZ}.
\begin{proposition}\label{constr}
Let $C$ be a left $A$-module which is Ext-injective in $\Cogen C$. Assume there are an injective cogenerator  $E$ of  $\lmod{A}$, a cardinal $\kappa$, and a right $\Prod C$-approximation $g:C^\kappa\to E$. Then $\Ker g\oplus C$ is a cosilting module with cosilting class $\Cogen C$.
\end{proposition}
\begin{proof}
Take a minimal injective copresentation  $0\to C\longrightarrow E_{0}\stackrel{\gamma}{\longrightarrow} E_1$. Then $\gamma^\kappa$ is an injective copresentation of $C^\kappa$. The map  $g:C^\kappa\to E$ induces a map $\tilde{g}:E_0\,^\kappa\to E$, which  can be viewed as a map of complexes $\tilde{g}:\gamma^\kappa\to E^\bullet$ where $E^\bullet$ is the complex with $E$ concentrated in degree 0.  Note that $\tilde{g}$ has zero cohomology $g$. Considering the mapping cone  $K_{\tilde{g}}: E_0\,^\kappa\stackrel{(\tilde{g},\gamma^\kappa)}{\to} E\oplus E_1\,^\kappa$ and setting $\omega=K_{\tilde{g}}[-1]$, we obtain a triangle in $\D(A)$ $$\omega\to \gamma^\kappa\stackrel{\tilde{g}}{\to} E^\bullet\to $$ whose zero cohomologies give rise to the exact sequence $$0\to C_1\to C^\kappa\stackrel{g}{\to} E\qquad\qquad$$ with $C_1=\Ker g=H^0(\omega)$.

We claim that $\Cogen C=\C_\omega$. Indeed, for a module $X$, applying $\Hom_{\D(A)}(X,-)$ to the triangle above yields a long exact sequence $$\Hom_{\D(A)}(X,\gamma^\kappa)\stackrel{\Hom_{\D(A)}(X,\tilde{g})}{\longrightarrow}\Hom_{\D(A)}(X,E^\bullet)\to \Hom_{\D(A)}(X,\omega[1])\to\Hom_{\D(A)}(X,\gamma^\kappa[1])\to 0.$$
Now, since  $g:C^\kappa\to E$ is a right $\Prod C$-approximation of an injective cogenerator, $X\in\Cogen C$ if and only if $\Hom_A(X,{g})$ is surjective, which amounts to $\Hom_{\D(A)}(X,\tilde{g})$ being surjective.
Further, recall from Lemma~\ref{corigid}(1) that $X\in\Ccal_\omega$ if and only if  $\Hom_{\D(A)}(X,\omega[1])=0$. This immediately gives the inclusion $\Cogen C\supset\Ccal_\omega$.
For the reverse inclusion, use that $\Hom_{\D(A)}(X,\gamma[1])=0$ for all $X\in\Cogen C$  by assumption and Lemma~\ref{corigid}(2).

In order to prove that $\tilde{C}=C_1\oplus C$ is a cosilting module, we now consider its injective copresentation $\tilde{\omega}=\omega\oplus \gamma$. By construction 
$\Ccal_\omega\subset\Ccal_{\gamma^\kappa}=\Ccal_\gamma$. Therefore $\Ccal_{\tilde{\omega}}=\Ccal_\omega\cap\Ccal_\gamma=\Ccal_\omega=\Cogen C=\Cogen\tilde{C}$.
\end{proof}

We know from  \cite{AIR} that $\tau$-tilting modules are
 ``non-faithful tilting'' modules. The same holds true for silting modules, as shown in \cite{AMV1}.
 Here is the dual case.

\begin{theorem}\cite{WZ}\label{cosilting}
The following statements are equivalent for an $A$-module $C$.
\begin{enumerate}
\item $C$ is a cosilting module.
\item $\Cogen C$ is a torsionfree class, and $C$ is cotilting over $\overline{A}=A/\Ann C$.
\item $C$ is Ext-injective in $\Cogen C$, and there are an injective cogenerator  $E$ of  $\lmod{A}$ and  an exact sequence $0\to C_1\to C_0\stackrel{g}{\to} E$ such that $C_1,C_0\in\Prod C$ and $g$ is a right $\Cogen C$-approximation.
\end{enumerate}
\end{theorem}
\begin{proof} 
First of all, notice that $\Ann C=\Ann{\Cogen C}$
and $\Cogen _{\overline{A}}C=\Cogen _AC$. Moreover, if $\Cogen C$ is extension closed (which holds true in all three statements), then $\Ext^1_{\overline{A}}(\Cogen C,C)=0$ if and only if $\Ext^1_A(\Cogen C,C)=0$. Finally, note also that  $\Ann C=\bigcap_{h\in\Hom_A(A,C)}\Ker h$, so the class $\Cogen _{\overline{A}}C$ contains $\overline{A}$ and all projective $\overline{A}$-modules.

(1) $\Rightarrow$ (2): 
Every cosilting module satisfies  $\Cogen C\subset {}^{\perp_1} C$, thus $\Cogen _{\overline{A}}C\subset \Ker\Ext^1_{\overline{A}}(-,C)$. To verify that $_{\overline{A}}C$ is cotilting, it remains to show the reverse inclusion. Pick a left $\overline{A}$-module $X$ with $\Ext^1_{\overline{A}}(X,C)=0$. Since $\Cogen C$ contains all projective $\overline{A}$-modules, there is a short exact sequence $0\to K\to L\to X\to 0$ in $\lmod{\overline{A}}$ with  $K,L\in\Cogen C$. Now one proves (dually to the proof of \cite[Proposition 3.10]{AMV1}) that the diagonal map $ K\hookrightarrow C^I$ with $I=\Hom_A(K,C)$ has cokernel $M\in\Cogen C$. By assumption on $X$, the  map  $\Hom_{\overline{A}}(L,C^I)\to \Hom_{\overline{A}}(K,C^I)$ is surjective, yielding a map $f$ and a commutative diagram
\begin{equation}
\xymatrix{0\ar[r]&K\ar @{=}[d]\ar[r]^{} & L\ar[d]^{f}\ar[r] & X\ar[d]^{g}\ar[r] & 0\\
0\ar[r] & K\ar[r]^{} & C^I\ar[r]^{} & M\ar[r] & 0.}
\end{equation}
Since $L,M\in\Cogen C$, we infer  $\Ker g\cong\Ker f\in \Cogen C$, and $\im g\in \Cogen C$. But $\Cogen C$ is closed under extensions, so also $X\in\Cogen C$.

(2) $\Rightarrow$ (3):  If $E$ is an injective cogenerator of $\lmod{A}$, then $E'=\{x\in E\mid \Ann C\cdot x=0\}$ is an injective cogenerator of $\lmod{\overline{A}}$. By Theorem~\ref{att} there is a 
short exact sequence of $\overline{A}$-modules $0\to  C_1\to C_0\stackrel{g'}{\to} E'\to 0$ with $C_0\in\Cogen C$ and $C_1\in \Ker\Ext^1_{\overline{A}}(\Cogen C,-)$. Then $g'$ is a right $\Cogen C$-approximation.
 Further, since $E'\in \Ker\Ext^1_{\overline{A}}(\Cogen C,-)$, both $C_1$ and $C_0$ belong to $\Cogen C\cap\Ker\Ext^1_{\overline{A}}(\Cogen C,-)=\Prod C$. 
 
Keeping in mind that every module in $\Cogen C$ is also a $\overline{A}$-module, we conclude that   the map $g:C_0\stackrel{g'}{\to} E'\subset E$ is a right $\Cogen C$-approximation with  the stated properties. Finally, $C$ is Ext-injective in $\Cogen C$ by the first paragraph  of the proof.

(3) $\Rightarrow$ (1): We can assume w.l.o.g. that $C_0=C^\kappa$ for some cardinal $\kappa$. By Proposition~\ref{constr}, there is an injective copresentation $\tilde{\omega}$ of the module $\tilde{C}=C_1\oplus C$ such that $\Ccal_{\tilde{\omega}}=\Cogen \tilde{C}=\Cogen C$. Take  minimal injective copresentations $\alpha$ and $\gamma$ of $C_1$ and $C$, respectively. Then $\alpha\oplus \gamma$ is a minimal injective copresentation of $\tilde{C}$, hence there are injective modules $I,I'$ such that $\tilde{\omega}=(\alpha\oplus \gamma)\oplus (0\to I)\oplus (I'\stackrel{\rm id}{\to} I')$.
Further, it is easy to see that $C_1\in\Prod C$ implies $\Ccal_\gamma\subset\Ccal_\alpha$, so $\Ccal_{{\alpha\oplus \gamma}}=\Ccal_\gamma$.
We infer  $\Ccal_{\tilde{\omega}}=\Ccal_\gamma\cap\Ccal_{(0\to I)}$.
So the injective copresentation 
 $\omega=\gamma\oplus (0\to I)$ of $C$ satisfies $\Ccal_\omega=\Ccal_{\tilde{\omega}}=\Cogen C$. This completes the proof.\end{proof}

It was shown in \cite{B} that every cotilting module is pure-injective, and cotilting classes are  definable. The characterisation  given in Theorem~\ref{cosilting}(2) now allows to deduce the same properties  for  cosilting modules. 
\begin{corollary}\cite{BP,WZ}\label{pi} Every cosilting module is pure-injective, every cosilting class is definable.
\end{corollary}

We are ready for a generalisation of Theorem~\ref{att} and \cite[Theorem 6.1]{Ba}.
\begin{theorem}\cite{WZ,BZ}\label{deftfree}
The following statements are equivalent for a  torsionfree class $\Fcal$ in $\lmod{A}$.
\begin{enumerate}
\item $\Fcal=\Cogen C$ for some cosilting module $C$.
\item For every left $A$-module $M$ there is an exact sequence $0\to K\to L\stackrel{g}{\to} M$ such that  $g$ is a right $\Fcal$-approximation and $K$ is Ext-injective in $\Fcal$.
\item Every left $A$-module admits a minimal right $\Fcal$-approximation.
\end{enumerate}
\end{theorem}
Let us sketch the proof. If  $\Fcal$ is definable, then every module admits a minimal right $\Fcal$-approximation, see \cite[Corollary 2.6]{Ba} and the references therein. Moreover, Wakamatsu's Lemma ensures that every minimal right $\Fcal$-approximation has an Ext-injective kernel. Therefore (1) $\Rightarrow$ (3) $\Rightarrow$ (2). The implication (2) $\Rightarrow$ (1) is shown by looking at the special case when  $M=E$ is an injective cogenerator of $\lmod{A}$. In this case  $\Fcal=\Cogen L=\Cogen C$ for $C=K\oplus L$, and  by  Theorem~\ref{cosilting}, the  module $C$ is  cosilting provided it is Ext-injective in $\Fcal$. So it remains  to verify Ext-injectivity of $L$.  To this end, one works over ${\overline{A}}=\Ann C=\Ann\Fcal$, where the injective cogenerator $E'=\{x\in E\mid \Ann C\cdot x=0\}$ admits a short exact sequence $0\to K\to L\to E'\to 0$, and $L$ is in $\Ker\Ext^1_{\overline{A}}(\Fcal,-)$ since so are $K$ and $E'$  (cf.~the proof of Theorem~\ref{cosilting}).

\begin{corollary}\label{bijection}
The assignment $C\mapsto \Cogen C$ defines a bijection between
\begin{enumerate}
\item[(i)] equivalence classes of cosilting left $A$-modules;
\item[(ii)] definable torsionfree classes in  $\lmod{A}$.
\end{enumerate} 
\end{corollary}

\smallskip

Notice that the statements above rely on the existence of {\em minimal} injective copresentations. Breaz and {\v Z}emli{\v c}ka show in \cite{BZ} that the dual version of Theorem~\ref{deftfree} holds true over perfect or hereditary rings. In general, however, one only has that the conditions dual to statements (2) and (3) in Theorem~\ref{cosilting} are equivalent  \cite[Proposition 3.2]{AMV1} and are satisfied by every silting module. In \cite[Proposition 3.10]{AMV1} this is phrased by saying that every silting module is {\em finendo quasitilting}.  But the converse is not true. A counterexample will be given in Example~\ref{notcoft}.

Moreover, there is a further asymmetry. Indeed,   definable classes give  rise to minimal right approximations, and also  to left approximations,  but not necessarily  to {\it minimal} left approximations. So not all silting (nor tilting)  classes satisfy the dual of  condition (3) in Theorem~\ref{deftfree}. Moreover, not all definable
torsion classes are generated by a finendo quasitilting module, an example is given in \cite[Proposition 7.2]{Ba}. However, we will see in the forthcoming section that definable torsion classes coincide with silting classes over noetherian rings.


\section{Duality} 
For a better understanding of the asymmetries discussed  above, 
we  need to review the interplay between definable subcategories of $\ModA$ and $\lmod{A}$.
Let us briefly recall that a subcategory $\mathcal{D}$ of $\ModA$ is definable if and only if it is the intersection of the kernels of a set of coherent functors $\ModA\to \Ab$, that is, of additive functors commuting with direct limits and direct products. Since the coherent functors are precisely the cokernels of morphisms of functors $\Hom_A(\sigma,-):\Hom_A(M,-)\longrightarrow\Hom_A(N,-)$  induced by morphisms $\sigma:N\to M$ in 
$\modA$, one obtains that $\Dcal$ is definable if and only if it is
of the form $\Dcal=\Dcal_\Sigma$ for a set $\Sigma$ of morphisms between finitely presented right $A$-modules, see e.g.~\cite[Sections 2.1 and 2.3]{CB}. 

Further, denoting by $(\modA,\Ab)$ and $(\lfmod{A},\Ab)$ the categories of additive covariant functors on $\modA$ and $\lfmod{A}$, respectively, one has an assignment  
$$(\modA,\Ab)\to (\lfmod{A},\Ab), F\mapsto F^{\vee}$$
where $F^\vee$ is defined on a module $N$ in $\lfmod{A}$ by
\[F^{\vee}(N)=\Hom(F, -\otimes N).\]
Notice that Hom-functors  $F=\Hom_A(M,-)$ given by $M\in\modA$ are mapped to $\otimes$-functors $F^\vee=-\otimes_A M$, see e.g.~\cite[Example 10.3.1]{P}.

This assignment was first studied by Auslander and by Gruson and Jensen. It 
induces a duality between the finitely presented objects in $(\modA,\Ab)$ and $(\lfmod{A},\Ab)$, and it allows to 
associate to every definable category  $\Dcal$ in $\ModA$ a \emph{dual definable category}  $\Dcal^\vee$ in $\lmod{A}$. 
Indeed,  $ \Dcal^\vee=\Fcal_\Sigma$ where $\Sigma$  is the maximal collection of   morphisms in $\modA$ such that $\Dcal=\Dcal_\Sigma$.
In this way one obtains a  bijection between definable subcategories of $\ModA$ and $\lmod{A}$ interchanging definable torsion classes  with definable torsionfree classes.  
For details we refer to \cite[Section 5]{Ba} and the references therein.

We now review some results from \cite{AH} explaining  how the assignment $\Dcal\mapsto\Dcal^\vee$ acts on the definable torsion or torsionfree classes given by silting and cosilting modules, respectively.
We fix a commutative ring $k$ such that $A$ is a $k$-algebra, and given an $A$-module $M$, we denote by $M^+$ its dual with respect to an injective cogenerator of $\rmod k$. For example, take $k=\Z$ and $M^+$ the character dual of $M$.

\begin{proposition}\label{dual} 
Let $\sigma \in \Mor{\ProjA}$.		
	\begin{enumerate}
\item   $\sigma^+$ is a morphism between injective left $A$-modules with $ \mathcal{C}_{\sigma^+} =\mathcal{F}_\sigma$.
\item  If $\Dcal_\sigma$ is a silting class, then $\Dcal_\sigma\,^\vee =\Fcal_{\sigma}$.	 
\item  If $T$ is a silting module with respect to $\sigma$, then $T^+$ is a cosilting module  with respect to   $\sigma^+$.
\end{enumerate}
		\end{proposition}
				
\begin{definition}\label{cofint} A cosilting left $A$-module $C$ (or the cosilting class $\Cogen C$) is said to be {\em of cofinite type} if there is a set $\Sigma\subset\Mor\projA$ such that $\Cogen C=\Fcal_\Sigma$.
		\end{definition}

\begin{corollary}\label{duality} The assignment $\Dcal\mapsto \Dcal^\vee$ defines a bijection between silting classes  in $\ModA$ and cosilting classes of cofinite type in $\lmod{A}$. \end{corollary}

If we restrict to tilting classes, that is, to classes $\Dcal=\Dcal_\Sigma$ with $\Sigma$ consisting of monomorphisms in $\Mor\projA$, then we recover the bijection between tilting classes and cotilting classes of cofinite type established in \cite{Ba}. Indeed, in this case $\Fcal_\Sigma =\Ker \Tor_1^A(\Scal, -)$
where $\Scal$ is the set of finitely presented right $A$-modules of projective dimension at most one that arise as cokernels of the monomorphisms in $\Sigma$. So Definition~\ref{cofint} agrees with the  definition  of a  cotilting class  of cofinite type in \cite{Ba}. 

\medskip

Here is a useful criterion for a torsionfree  class to be  cosilting of cofinite type.
\begin{lemma}\label{coft}
A torsion pair $(\mathcal{T},\mathcal{F})$ in $\lmod{A}$ is generated by  a set of finitely presented left $A$-modules if and only if $\mathcal{F}$ is a cosilting class of cofinite type.
\end{lemma}
\begin{proof} Let  $\mathcal{U}$ be a set in $\lfmod{A}$ such that $\mathcal{F}=\{M\in \lmod{A}\,\mid\, \Hom_{A}({U},{M})=0 \text{ for all } U\in \mathcal{U}\}.$ 
 Choosing a projective presentation
 $\alpha_U\in\Mor\Aproj$ for each $U\in\mathcal{U}$ and applying $^\ast =\Hom_A(-,A)$ on it, we obtain a set $\Sigma=\{\alpha_U\,^\ast\,\mid\,U\in\Ucal\}\subset\Mor\projA$
such that $\mathcal{F}= \mathcal{F}_{\Sigma}.$
The other implication is proven similarly.
\end{proof}

\medskip

Now recall  that over a left noetherian ring every torsion pair  $(\mathcal{T},\mathcal{F})$  in $\lmod{A}$ restricts to a torsion pair $(\mathcal{U},\mathcal{V})$ in $\lfmod{A}$ with $\Ucal=\Tcal\cap\lfmod{A}$ and  $\Vcal=\Fcal\cap\lfmod{A}$, and moreover, taking direct limit closures, the latter torsion pair $(\mathcal{U},\mathcal{V})$ extends to a torsion pair $(\varinjlim\mathcal{U},\varinjlim\mathcal{V})$
in $\lmod{A}$, see \cite[Lemma 4.4]{CB1}. So, if
 $\Fcal$  is definable, that is, closed under direct limits, we see that  $(\mathcal{T},\mathcal{F})=(\varinjlim\mathcal{U},\varinjlim\mathcal{V})$ is generated by $\Ucal$. Using Lemma~\ref{coft}, we  conclude that every definable torsionfree class in $\lmod{A}$
is a  cosilting class of cofinite type. This also implies that
every definable torsion class in $\ModA$ is a  silting class. 

\begin{theorem}\cite[Corollary 3.8]{AH}\label{coftnoe}
 If $A$ is a left noetherian ring,  the assignment $\Dcal\mapsto \Dcal^\vee$ defines a bijection between silting classes  in $\ModA$ and cosilting classes in $\lmod{A}$, and there are bijections 
  between
\begin{enumerate}
\item[(i)]  equivalence classes of  silting right $A$-modules;
\item[(ii)] equivalence classes of  cosilting left $A$-modules;
\item[(iii)] definable torsion classes in $\ModA$;
\item[(iv)]  definable torsionfree classes in $\lmod{A}$.
\end{enumerate}
The bijection  $(i)\to (iii)$ is given by the assignment $T\mapsto \Gen T$, the bijection $(ii)\to (iv)$ is defined dually.
\end{theorem}

The theorem above   can be regarded as a ``large'' version  of the following result for  finite dimensional algebras due to Adachi-Iyama-Reiten.
\begin{theorem}\cite[Theorem 2.7]{AIR}\label{AIR} If $A$ is a finite dimensional algebra over a field, there are bijections between
\begin{enumerate}
\item[(i)]  isomorphism classes of basic support $\tau$-tilting (i.e.~finite dimensional silting) right $A$-modules;
\item[(ii)] isomorphism classes of basic finite dimensional cosilting left $A$-modules;
\item[(iii)] functorially finite torsion classes in $\modA$;
\item[(iv)]  functorially finite torsionfree classes in $\lfmod{A}.$
\end{enumerate}
The bijection  $(i)\to (iii)$ is given by the assignment $T\mapsto \gen T=\Gen T\cap\modA$, the bijection $(ii)\to (iv)$ is defined dually.
\end{theorem}

Next, we apply the criterion in Lemma~\ref{coft} to hereditary torsion pairs. Recall that a torsion pair $(\mathcal{T},\mathcal{F})$ is {\it hereditary} if the torsion class $\mathcal{T}$ is closed under submodules, or equivalently, the torsionfree class $\mathcal{F}$ is closed under injective envelopes.  
Moreover, $(\mathcal{T},\mathcal{F})$ has {\em finite type} if $\mathcal{F}$ is closed under direct limits.

Hereditary torsion pairs are in bijection with the Gabriel topologies on $A$. We refer to \cite[Ch.VI]{Ste} or \cite[Sections 11.1.1 and 11.1.2]{P} for details. Here we only mention that every hereditary torsion pair $(\mathcal{T},\mathcal{F})$ in $\lmod{A}$ is associated to a   {\em Gabriel filter} $\Gcal$ which  is formed  by  the left ideals $I$ of $A$ such that $A/I\in\Tcal$. When the torsion pair is of finite type, then $\Gcal$ has a basis of finitely generated ideals, that is, every ideal in $\Gcal$ contains a finitely generated ideal from $\Gcal$. This implies that the torsion pair $(\mathcal{T},\mathcal{F})$ is generated by the set $\Ucal=\{A/I\,\mid\, I\in\Gcal\text{ finitely generated }\}\subset\lfmod{A}$, and so $\mathcal{F}$ is a cosilting class  of  cofinite type.
The dual definable category $\Fcal^\vee$ is then given by the right $A$-modules with $M\otimes A/I=0$ for all finitely generated ideals $I\in\Gcal$, and one easily checks that this amounts to $M$ being $I$-{\em divisible}, i.e.~$MI=M$, for all such $I$. We will say that  $\Fcal^\vee$ is the {\em class of divisibility} by these ideals. 
\begin{corollary}\label{hercos}
The assignment $\Dcal\mapsto \Dcal^\vee$ restricts to a bijection between the silting classes occurring as classes of divisibility by sets of finitely generated left ideals, and the torsionfree classes in hereditary torsion pairs of finite type.
\end{corollary} 

Over a commutative ring,  every cosilting   class of cofinite type arises from a hereditary torsion pair of finite type, because it turns out that every  torsionfree class of the form $\Fcal_\sigma$ for some $\sigma\in\Mor\projA$ is closed under injective envelopes \cite[Lemma 4.2]{AH}. This  yields the following classification result.

\begin{theorem}\cite[Theorems 4.7 and 5.1]{AH}\label{classcomm}
If $A$ is commutative ring, there is a bijection between
	\begin{enumerate}
			\item[(i)] equivalence classes of silting $A$-modules,
			\item[(ii)] hereditary torsion pairs of finite type in $\ModA$.
	\end{enumerate}
	If $A$ is commutative noetherian, there are further bijections with
	\begin{enumerate}
	\item[(iii)] equivalence classes of cosilting $A$-modules,
	\item[(iv)] subsets $P \subseteq \Spec A$ closed under specialisation.
	\end{enumerate}
\end{theorem}
The bijection between (ii) and (iv) is well known, see  \cite[Chapter VI, \S 6.6]{Ste}. 
An explicit construction of a silting module in (i) and a cosilting module in (iii) is provided in \cite{AH}.

\smallskip

We close this section with an example of a cosilting module not of cofinite type.
\begin{lemma}
Let $A$ be a ring, and $S$ a simple module such that $S$ is a finitely generated module over $\End_AS$ and  $\Ext^1_A(S,S)=0$.
Then $S$ satisfies  condition (3) in Theorem~\ref{cosilting} and its dual. In particular, $S$ is a cosilting module. 
\end{lemma}
\begin{proof}
Since $\End_AS$ is a skew-field, the module $S$ has finite length over $\End_AS$, and therefore $\Add S=\Prod S$, see \cite{KS}. By assumption, $\Gen S=\Add S=\Cogen S$ are contained both in $S^{\perp_1}$ and ${}^{\perp_1} S$. Moreover, if $E$ is an injective cogenerator of $\ModA$, the codiagonal map
$S^{(I)}\to E$ given by $I=\Hom_A(S,E)$ is a right $\Cogen S$-approximation with kernel in $\Prod S$, yielding condition (3) in Theorem~\ref{cosilting}. For the dual condition take the diagonal map 
$A\to S^J$ given by $J=\Hom_A(A,S)$.
\end{proof}
\begin{example}\cite[Example 5.4]{AH}\label{notcoft}
Let $(A,\mathfrak m)$ be a valuation domain whose maximal ideal $\mathfrak m=\mathfrak m^2$ is idempotent and non-zero.
Then $S=A/\mathfrak m$ satisfies the conditions of the Lemma above (indeed, we will see in Example~ \ref{exs2}(2) and Theorem~\ref{epicl}(2) that the idempotency of $\mathfrak m$ implies $\Ext^1_A(S,S)=0$), so it is a cosilting module with $\Cogen S=\Gen S=\Add S=\{M\in\ModA\,\mid\,M\mathfrak m=0\}$.

Notice that $\Cogen S$ does not contain the injective envelope of $S$, and so it does not arise from a hereditary torsion pair. It follows from Theorem~\ref{classcomm} that $S$ is not a cosilting module of cofinite type.
Moreover, despite the fact that $S$ satisfies the statement dual to condition (3) in Theorem~\ref{cosilting}, it is not a silting module. This follows again from Theorem~\ref{classcomm}, because  if 
$\Gen S$ were a silting class, it would  be the class of divisibility by a set of finitely generated left ideals. But the only ideal $I$ with $SI=S$ is $I=A$, which would entail $\Gen S=\ModA$, a contradiction.  
\end{example}

\section{Ring epimorphisms}
We have seen above that silting modules are closely related with localisation at Gabriel topologies, and silting classes are often given by divisibility conditions. In this section, we discuss the connections between silting and ring theoretic localisation. We first recall the relevant terminology.

\begin{definition}\cite[Ch.XI]{Ste}\cite{GL}
A {ring homomorphism} $f:A\rightarrow B$  is a 
 {\em ring epimorphism} if it is an epimorphism in the category of rings with unit, or equivalently, if the functor given by restriction of scalars  $f_\ast:\ModB\rightarrow \ModA$ is a full embedding.
 
Further, $f$ is a 
{\em homological ring epimorphism} if it is a ring epimorphism and $\Tor_i^A(B,B)=0$ for all $i>0$, or equivalently,  the functor given by restriction of scalars  $f_\ast:\D(B)\rightarrow \D(A)$ is a full embedding.

Finally, $f$ is a 
{\em (right) flat ring epimorphism} if it is a ring epimorphism and $B$ is a flat right $A$-module.
\end{definition}

 Two ring epimorphisms $f_1:A\rightarrow B_1$ and  $f_2:A\rightarrow B_2$ are  {\em equivalent} if there is a ring isomorphism $h: B_1\rightarrow B_2$ such that $f_2=h\circ f_1$. We then say that they lie in the same {\em epiclass} of $A$.

\begin{definition}\cite{GL},\cite[Theorem 1.6.3]{I} A full subcategory $\Xcal$ of $\ModA$ is called {\em bireflective} if the inclusion functor $\Xcal\hookrightarrow\ModA$ admits both a left and right adjoint, or equivalently, if it is closed under products, coproducts, kernels and cokernels.

Moreover, a full subcategory $\Wcal$ of $\modA$ is said to be {\em wide} if it is an abelian subcategory of $\modA$ closed under extensions.
\end{definition}

\begin{theorem}\label{epicl} (1) \cite{GdP} There is a bijection between
\begin{enumerate}
\item[(i)] epiclasses of ring epimorphisms;
\item[(ii)] bireflective subcategories of $\ModA$.
\end{enumerate}
It assigns to a ring epimorphism  $f:A\ra B$  the essential image $\Xcal_B$ of the restriction functor $f_\ast:\ModB\hookrightarrow \ModA$.

(2) \cite{BD}\cite[Theorem 4.8]{Sch} The following statements are equivalent for  a ring epimorphism $f:A\ra B$: \begin{enumerate}
\item $\Xcal_B$ is closed under extensions in $\ModA$;
\item $\Tor_1^A(B,B)=0$;
\item the functors $\Ext^1_A$ and $\Ext^1_B$ coincide on $B$-modules;
\item the functors $\Tor_1^A$ and $\Tor_1^B$ coincide on $B$-modules. 
\end{enumerate}

(3) \cite{GdP,GL}, \cite[Theorem 1.6.1]{I} If $A$ is a finite dimensional algebra,  the assignment $f\mapsto \Xcal_B\cap\modA$ defines a bijection between
\begin{enumerate}
\item[(i)] \mbox{epiclasses of ring epimorphisms $f:A\to B$ with $B$  finite dimensional  and $\Tor_1^A(B,B)=0$}; 
\item[(ii)] functorially finite wide subcategories of $\modA$. 
\end{enumerate}
\end{theorem}

The following result provides a large supply of ring epimorphisms with the properties in Theorem~\ref{epicl}(2).
\begin{theorem}\cite[Theorem~4.1]{Sch}\label{def:universallocalisation}
Let $A$ be a ring and $\Sigma$ be a set of morphisms in  $\Mor\projA$. Then
there is a ring homomorphism 
$f: A\rightarrow A_\Sigma$, called
\emph{universal localisation} of $A$ at
$\Sigma$, such that
\begin{enumerate}
\item[(i)] $f$ is \emph{$\Sigma$-inverting,} i.e.~$\sigma\otimes_A A_\Sigma$ is an isomorphism  for every
$\sigma$ in  $\Sigma$,  and 
\item[(ii)] $f$ is \emph{universal
$\Sigma$-inverting}, i.e.~for any $\Sigma$-inverting morphism $f': A\rightarrow B$
there exists a unique ring homomorphism ${g}:
A_\Sigma\rightarrow B$ such that $g\circ f=f'$.
\end{enumerate}
Moreover, $f\colon A\rightarrow A_\Sigma$ is a ring epimorphism
with  $Tor_1^{A}(A_\Sigma,{A_\Sigma})=0$.
\end{theorem}
Notice that $A_\Sigma$ coincides with the universal localisation at   $\Sigma ^\ast=\{ \sigma^\ast\,\mid\,\sigma\in\Sigma\}\subset\Mor{\Aproj}$. Moreover, the essential image of the restriction functor along $f\colon A\rightarrow A_\Sigma$ consists of the right $A$-modules such that
$X\otimes_A\sigma^\ast$ is an isomorphism  for every $\sigma\in\Sigma$, and it equals $$\Xcal_\Sigma=\{X\in \ModA \,\mid\, \Hom_A(\sigma,X)\ \text{is bijective for all } \sigma\in\Sigma \}.$$

\medskip

We are going to discuss how classes of the latter shape are related with silting modules.

\begin{definition} A module $T_1$ (the reason for this notation will become clear later) is called a {\em partial silting module} if  it admits a projective presentation $ P_{-1}\stackrel{\sigma}{\longrightarrow} P_0\to T_1\to 0$ such that 
$\Dcal_\sigma$ is a torsion class containing $T_1$. \end{definition}
 
Extending \cite[Theorem 2.10]{AIR}, it was shown in  \cite[Theorem 3.12]{AMV1} that every partial silting module $T_1$ can be completed to a silting module ${T}$ which is called a {\em Bongartz completion} of $T_1$ and generates the same torsion class $\Dcal_\sigma$. 
Moreover,  $\Gen T_1\subset \Dcal_\sigma\subset T_1\,^{\perp_1}$, and $(\Gen T_1, T_1\,^o)$ is a torsion pair. Then  $$\Xcal_\sigma=\Dcal_\sigma\cap(\Coker\sigma)^o=\Gen T\cap T_1\,^o.$$ 

As shown in \cite[Proposition 3.3]{AMV2}, the class $\Xcal_\sigma$ is bireflective and extension closed. Thus  it can  be realised as $\Xcal_B$ for some  ring epimorphism $f:A\to B$ as in Theorem~\ref{epicl}(2), and by \cite[Theorem 3.5]{AMV2} the ring $B$ can be described as an idempotent  quotient of  $\End_A{{T}}$.

A ring epimorphism  arising from a partial silting module as above will be called {\em silting ring epimorphism}.
\begin{theorem}\label{silt} \cite{MS2} Every universal localisation is a silting ring epimorphism.\end{theorem}

\begin{examples}\label{exs2}
(1)  A module $T_1$ is partial silting with respect to an monomorphic projective presentation $\sigma:P_{-1}{\hookrightarrow} P_0$ if and only if it is partial tilting, and in this case $\Xcal_\sigma=T_1^{\perp_1}\cap T_1\,^o$ is the perpendicular category of $T_1$ studied in \cite{GL,CTT}.

\smallskip

(2)
For any ideal $I$  of $A$, the canonical surjection $f:A\to \overline{A}=A/I$ is a ring epimorphism. Moreover, since $I/I^2\cong\Tor_1^A(\overline{A},\overline{A})$, the ideal $I=I^2$ is idempotent if and only if $\Xcal_{\overline{A}}$ is closed under extensions, in which case $\Xcal_{\overline{A}}=I^o$. 

An important example of an idempotent ideal is provided by the trace ideal  $I=\tau_P(A)$   of a projective right $A$-module $P$. Notice that every idempotent ideal $I$ has the form $I=AeA$  for some $e=e^2\in A$ whenever $A$ is a (one-sided) perfect ring \cite[Proposition 2.1]{Michler}, or when $A$ is commutative and $I$ is finitely generated, in which case $f$ is even a split epimorphism  \cite[Lemma 2.43]{Lam}. Moreover,  every idempotent ideal $I$ with $_AI$ being finitely generated is the trace ideal     of a countably generated projective right $A$-module, see \cite{Whitehead}.

\smallskip

(3) Let us focus on the case when $I=\tau_P(A)$ is the trace ideal of a projective right $A$-module $P$,  and $f:A\to \overline{A}=A/I$. In this case  $\Xcal_{\overline{A}}$ consists of the modules $M$ with $\tau_P(M)=MI=0$, so it is the perpendicular category of the partial tilting module $P$, and $f$ is therefore a silting ring epimorphism.
 In fact, $f$ is even a universal localisation.
This is clear when $P$ is finitely generated (then $f$ is  the universal localisation at $\Sigma=\{ 0\to P\}$). For the general case, one uses the following argument due to Pavel P\v{r}\'ihoda.  First of all, keeping in mind that  every projective module is a direct sum of countably generated projectives by a celebrated result of Kaplansky, we can assume w.l.o.g. that $P$ is countably generated. Then $P$ can be written as a direct limit of a direct system of finitely generated free modules $F_1\stackrel{\alpha_1}{\to} F_2\stackrel{\alpha_2}{\to}F_3\stackrel{\alpha_3}{\to} \ldots$
where each map $\alpha_i$ is given by multiplication with a matrix $X_i$ having its entries in the trace ideal $I$, and moreover, for each $i> 1$ there is a map $\beta_i:F_{i+1}\to F_i$ such that $\beta_{i}\alpha_{i}\alpha_{i-1}=\alpha_{i-1},$
see \cite[Theorem 1.9]{Whitehead}, \cite[Proposition 1.4]{HP}. Now we set $\Sigma=\{ 1_{F_i} - \beta_i\alpha_i\mid i> 1\}$.
Since $f$ is $\Sigma$-inverting,   the universal property in Theorem~\ref{def:universallocalisation} implies that $\Xcal_{\overline{A}}\subset\Xcal_\Sigma$. Conversely, the fact that $A\to A_\Sigma$ is $\Sigma$-inverting entails  $\alpha_i\otimes_A A_\Sigma=0$ for all $i\ge 1$, hence $P\otimes_A A_\Sigma=0$. But then $\Hom_A(P,A_\Sigma)\cong\Hom_{A_\Sigma}(P\otimes_A A_\Sigma, A_\Sigma)=0$, that is, $A_\Sigma$ belongs to $\Xcal_{\overline{A}}$, and so do all $A_\Sigma$-modules.

\smallskip

(4) If $A$ is a semihereditary ring, then all flat ring epimorphisms $A\to B$ are universal localisations \cite[Proposition 5.3]{angarc}, and the converse is true  if $A$ is also commutative \cite[Theorem 7.8]{BS}. Moreover, $\Tor_i^A$ vanishes for $i\ge 2$ (see e.g.~\cite[Theorem 4.67]{Lam}), hence every epimorphism as in Theorem~\ref{epicl}(2) is homological, and its kernel is an idempotent ideal by \cite[Lemma 4.5]{BS}.
 On the other hand, for a semihereditary commutative ring, the fact that every homological epimorphism  is a universal localisation amounts to the validity of the Telescope Conjecture for $\D(A)$, and it  fails in general, see \cite[Section 8]{BS} and the example  below.

\smallskip

(5) Let $(A,\mathfrak m)$ be as in Example~\ref{notcoft}. Then  $f:A\to \overline{A}=A/\mathfrak m$ is a homological ring epimorphism.  But $f$ is not a silting ring epimorphism, and thus not a universal localisation. Indeed, if there were a partial silting module $T_1$ with Bongartz completion $T$ such $\Xcal_{ \overline{A}}=\Gen T\cap T_1\,^o$, then by Theorem~\ref{classcomm} the silting class $\Gen T$ would be the class of divisibility by a set of finitely generated ideals of $A$, and then also by all ideals in the  corresponding Gabriel filter $\Gcal$. Since $\mathfrak m$ is the unique maximal ideal and $\overline{A}\in\Gen T$ is certainly not $\mathfrak m$-divisible, we infer that $\Gcal$ can only contain the ideal $A$. But then $\Gen T=\ModA$, and $T$ and $T_1$ are projective, hence free, thus $\Gen T\cap T_1\,^o=0$ cannot coincide with $\Xcal_{ \overline{A}}$.

\smallskip

(6) 
 If $A$ is a hereditary ring, then  homological ring epimorphisms and universal localisations of $A$ coincide \cite{KSt}. 
 Moreover, by \cite[Theorem 2.3]{Scho},
 there is a bijection between wide subcategories of $\modA$ and universal localisations of $A$, which maps a wide subcategory $\Wcal$ to the universal localisation at (projective resolutions of the modules in) $\Wcal$. Conversely, every $A\to A_\Sigma$ is associated to the wide subcategory $\Wcal$ formed by the $\Sigma$-{\em trivial}  modules, that is, the modules $M\in\modA$ admitting a projective resolution $\sigma:P_{-1}\hookrightarrow P_0$ with $\sigma\otimes_A A_\Sigma$ being an isomorphism.
 
 \smallskip
 
 (7) It follows from (6) that silting ring epimorphisms and universal localisations coincide for  hereditary rings. The same holds true for commutative noetherian rings of Krull-dimension at most one, but it fails already in  Krull-dimension  two, see \cite{MS2,AMSTV}.  In particular,  there is no  analog of Theorem~\ref{deft} for classes $\Xcal_\sigma$: in general, the map $\sigma\in\Mor{\ProjA}$ cannot be replaced by a set $\Sigma\subset\Mor{\projA}$. However, it is  shown in \cite{MS2} that one can always find a set $\Sigma$ of morphisms between {\em countably} generated projective modules such that $\Xcal_\sigma=\Xcal_\Sigma$.
 \end{examples}

\medskip

Our next aim is to investigate the relationship between silting modules and silting ring epimorphisms. We will need the following  construction of silting modules which is dual to Proposition~\ref{constr}.
\begin{proposition}\cite{AMV4}\label{Prop silting}
Let $T$ be a module with a projective presentation $P_{-1}\stackrel{\sigma}{\longrightarrow}  P_0$ such that $\Hom_{\mathsf{D}(A)}(\sigma,\sigma^{(I)}[1])=0$ for any set $I$. Assume there are a cardinal $\kappa$ and a left $\Add T$-approximation $f:A\longrightarrow T^{(\kappa)}$. Then $T\oplus \Coker f$ is a silting $A$-module with silting class $\Gen T$. Moreover, if $A {\longrightarrow}  \,{\sigma^{(\kappa)}} \longrightarrow \omega \longrightarrow$ is the triangle in $\mathsf{D}(A)$ induced by $f$, then $\Coker f$ is a partial silting module with respect to  $\omega$ and 
$\Gen{T}=\Dcal_{\omega}$.
\end{proposition}

Now, if  $T$ is a silting module with respect to a projective presentation $\sigma$, then by \cite{Wei} there is a triangle
\begin{equation}\label{approx tria}
A\stackrel{\phi}{\longrightarrow}\sigma_0\longrightarrow\sigma_1\longrightarrow
\end{equation}
in $\mathsf{D}(A)$, where $\sigma_0$ and $\sigma_1$ lie in $\Add{\sigma}$ and $\phi$ is a left $\Add{\sigma}$-approximation of $A$. Applying the functor $H^0(-)$ to this triangle, we obtain  a left $\Gen T$-approximation sequence for $A$ in $\ModA$  
\begin{equation}\label{approx seq}
A\stackrel{f}{\longrightarrow}T_0\longrightarrow T_1\longrightarrow 0
\end{equation}
where   $T_0,T_1\in\Add T$ and $T_0$ satisfies the assumptions of Proposition~\ref{Prop silting}, so $T_1$ is a partial silting module with respect to $\sigma_1$. 
Moreover,   $f$ is left minimal if so is $\phi$.
\begin{definition} 
Let $T$ be a silting module with respect to $\sigma$. If the map $\phi$ in the triangle (\ref{approx tria}) above can be chosen left-minimal, 
then $T$ is said to be a {\em minimal} silting module. 
\end{definition}
\begin{examples} \label{arising}
(1) If  $A$ is (right) hereditary or perfect, then any approximation sequence as in (\ref{approx seq}) can be lifted to a triangle as in (\ref{approx tria}), and a silting module  is minimal if and only if    the map $f$ in (\ref{approx seq}) can be chosen left-minimal, compare with \cite[Definition 5.4]{AMV2}.

\smallskip

(2) Every finite dimensional silting module over a finite dimensional algebra is minimal.

\smallskip

(3) The following  example extends  a construction of tilting modules from \cite{AS1}.
If $f:A\longrightarrow B$ is a homological ring epimorphism such that $B$ is an $A$-module of projective dimension at most one, then $B\oplus \Coker f$ is a minimal silting $A$-module. This follows immediately from  Proposition~\ref{Prop silting} keeping in mind that $f$ can always be regarded as a  minimal left $\Add B$-approximation of $A$, in fact, it is the $\Xcal_B$-reflection of $A$. Moreover,  $\Coker f$ is  partial silting  with respect to a projective presentation  $\omega$  such that $\Gen B=\Dcal_\omega$, and $f$ is the corresponding silting ring epimorphism, because 
$\Xcal_{\omega}=\Gen B\cap(\Coker f)^o=\Xcal_B$.
\end{examples}

The importance of minimal silting modules is due to the fact that an approximation triangle (\ref{approx tria}) with $\phi$ being left minimal is unique  up to isomorphism, and so is the module $T_1$. We can thus associate to $T$ a uniquely determined silting ring epimorphism $f:A\to B$ 
with $\Xcal_B=\Xcal_{\sigma_1}$.

Notice that the class $\Xcal_{\sigma_1}$ can also be described in a different way. For any torsion class $\Tcal$ in an abelian category $\Acal$, we  consider the subcategory   of $\Acal$ 
\begin{equation}\label{alpha}\mathfrak a(\Tcal):=\{X\in\Tcal: \text{ if } (g:Y\rightarrow X)\in\Mor\Tcal, \text{ then } \Ker(g)\in\Tcal\}\end{equation} studied in  \cite{IT}.  
It turns out that in our situation $\Xcal_{\sigma_1}=\mathfrak a(\Gen T)$, see \cite[Remark 5.7]{AMV2}. In summary:
\begin{proposition}\cite{AMV4}\label{inj1}
There is a commutative diagram 
$$\xymatrix{{\left\{\begin{array}{c}\text{\Small equivalence classes}\\ \text{\Small of minimal}\\ \text{\Small silting $A$-modules} \end{array}\right\}}\ar[rr]^{\alpha}\ar[dr]^{\mathfrak{a}\,\,\,} &  & {\left\{\begin{array}{c}\text{\Small epiclasses of ring} \\ \text{\Small epimorphisms $A\to B$}\\ \text{\Small with $\Tor_1^A(B,B)=0$} \end{array}\right\}}\ar[dl]_{\,\,\,\epsilon}\\ & {\left\{\begin{array}{c}\text{\Small  bireflective }\\ \text{\Small extension-closed} \\ \text{\Small subcategories}\\ \text{\Small of $\ModA$}\end{array}\right\}} &}$$
where $\alpha$ assigns to a  silting module $T$ the associated silting ring epimorphism, the map 
$\epsilon$ is the bijection from Theorem \ref{epicl}(1) and
$\mathfrak{a}$  is  the map defined above.
\end{proposition}

Furthermore, if $A$ is hereditary, then $T_0$ is a projective generator of $\mathfrak a(\Gen T)$, and the category of projective $B$-modules  is equivalent to $\Add T_0$, see \cite[Proposition 5.6]{AMV2}.  This shows that the epiclass of  $f$ determines $\Add T_0$ and thus also the silting class $\Gen T_0=\Gen T$. Hence  $\alpha$ is an injective map. 
On the other hand, we also have a reverse assignment. According to Example~\ref{arising}(3), to every homological ring epimorphism $f:A\to B$ we can associate the minimal silting module $T=B\oplus \Coker f$, and taking the silting ring epimorphism corresponding to $T$ we recover $f$. 

Combined with the results from \cite{Scho,KSt} explained in Example~\ref{exs2}(6), we obtain

\begin{theorem}\cite{AMV2}\label{her}
Let $A$ be a hereditary ring. There is a bijection between 
\begin{enumerate}
\item[(i)] equivalence classes of minimal silting $A$-modules;
\item[(ii)] epiclasses of homological ring epimorphisms of $A$;
\item[(ii')] epiclasses of universal localisations of $A$;
\item[(iii)] wide subcategories of $\modA$.
\end{enumerate}
Under this bijection,  epiclasses of injective homological ring epimorphisms of $A$ correspond to
 minimal tilting $A$-modules.
 \end{theorem}

\begin{example}
If $A$ is a commutative hereditary ring, then every silting module is minimal.
Indeed, by Theorem~\ref{classcomm} every silting module $T$  corresponds to a hereditary torsion pair of finite type, that is, to a  Gabriel filter $\Gcal$ with a basis of finitely generated ideals. Since all ideals are projective, $\Gcal$ is a {\it perfect} Gabriel topology according to  \cite[Chapter XI, Proposition 3.3 and Theorem 3.4]{Ste}, and it induces a flat ring epimorphism $f:A\to B$ satisfying the assumptions of Example~\ref{arising}(3). Then one can show as in \cite[Theorem 5.4]{Hrbek} that $T$ is equivalent to the minimal silting module $B\oplus\Coker f$. 
 
In particular, if $A$ is a Dedekind domain,   there is a bijection between equivalence classes of tilting modules and sets of maximal ideals, and the trivial module 0 is the only silting module that is not tilting, see \cite[Corollary 6.12]{AS1}. 
\end{example}

\section{The lattice of ring epimorphisms}

The partial order  on bireflective subcategories   given by inclusion corresponds under the bijection in Theorem~\ref{epicl}(1) to the   partial order on the epiclasses of $A$ defined by setting   $$f_1\ge  f_2$$  whenever there is a commutative diagram of ring homomorphisms
$$\xymatrix{A\ar[rr]^{f_2}\ar[dr]_{f_1} & & B_2\\ & B_1\ar[ur]_{g} & }$$
Since bireflective subcategories are determined by closure properties,
 the poset induced by $\ge$ is a lattice, and  the ring epimorphisms $A\to B$ with $\Tor_1^A(B,B)=0$ form a sublattice in it by Theorem~\ref{epicl}(2).
 We now compare these posets with other posets recently studied in cluster theory.

\medskip

First of all, observe that
over a finite dimensional algebra,  Theorem~\ref{her} has also a ``small'' version.

\begin{corollary}\cite[Section 2]{IT},\cite[Theorem 4.2]{Ma}\label{IngallsThomas} 
If $A$ is a finite dimensional hereditary algebra, there are bijections between 
\begin{enumerate}
\item[(i)] isomorphism classes of basic support tilting (i.e.~finite dimensional silting) modules;
\item[(ii)] epiclasses of homological ring epimorphisms  $A\to B$ with $B$ finite dimensional;
\item[(iii)] functorially finite wide subcategories of $\modA$.
\end{enumerate}
\end{corollary}

In \cite{IT,IgS,R} further bijections are established, providing a combinatorial interpretation of finitely generated silting modules in terms of noncrossing partitions,   clusters, or antichains.
Observe that the poset of noncrossing partitions corresponds to the poset given by $\ge$ on the epiclasses  in condition (ii), and it does not form a lattice in general (unless we relax the condition that $B$ is finite dimensional, as discussed above).

\begin{examples}\label{posets}
(1) Let $A$ be a tame hereditary finite dimensional algebra with a tube $\mathcal U$ of rank $2$, and let $S_1,S_2$ be the simple regular modules in $\mathcal U$. The universal localisation $f_i:A\to A_{\{S_i\}}$ of $A$ at (a projective resolution of) $S_i$ has a finite dimensional codomain  $A_{\{S_i\}}$   for $i=1,2$.
The meet of $f_1$ and $ f_2$, however, is the universal localisation $A\to A_{\{S_1,S_2\}}$  of $A$ at $\mathcal U$ and $A_{\{S_1,S_2\}}$ is an infinite dimensional algebra by  \cite[Theorem 4.2]{CBreg}.
\smallskip

(2) Let $A$ be the Kronecker algebra, i.e.~the path algebra of the quiver $\xy\xymatrixcolsep{2pc}\xymatrix{ \bullet \ar@<0.5ex>[r]  \ar@<-0.5ex>[r] & \bullet } \endxy$ over an algebraically closed field $\mathbb K$. The lattice of universal localisations  has the following shape

\smallskip

$$
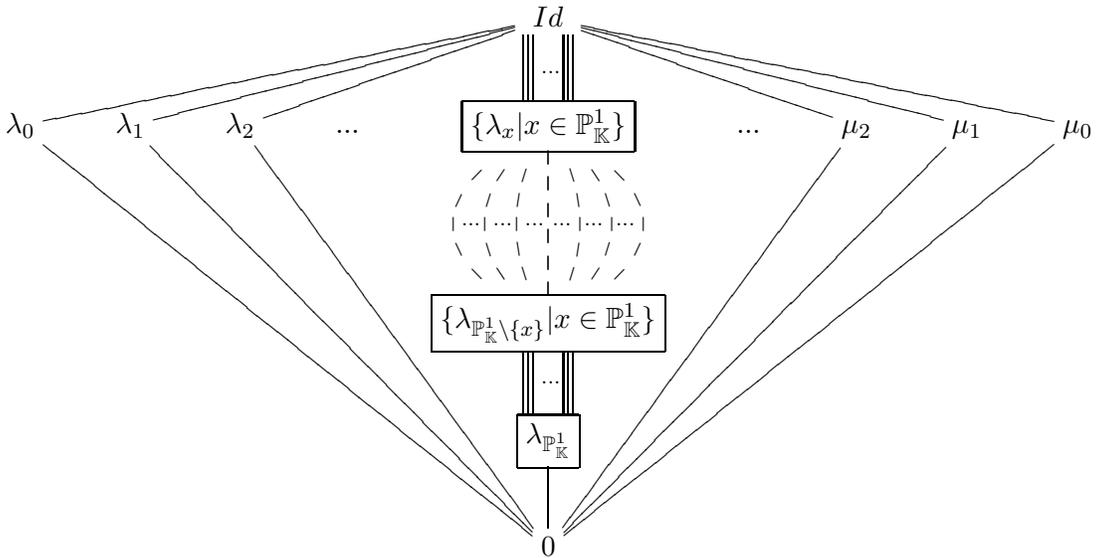
\xymatrix{&&&& \ Id \ \ar@{-}[lllld]\ar@{-}[llld]\ar@{-}[lld]\ar@{-}[rrrrd]\ar@{-}[rrrd]\ar@{-}[rrd]\ar@<-2ex>@{-}[d]^{\ ...}\ar@<-1.6ex>@{-}[d]\ar@<-1.2ex>@{-}[d]\ar@<1.2ex>@{-}[d]\ar@<1.6ex>@{-}[d]\ar@<2ex>@{-}[d]\\ \lambda_0\ar@{-}[rrrrdddd]&\lambda_1\ar@{-}[rrrdddd]&\lambda_2\ar@{-}[rrdddd] & ... & *+[F]{\{\lambda_x| x\in \mathbb{P}^1_\mathbb{K}\}}\ar@{--}@/_1pc/[dd]^{...}\ar@{--}[dd]\ar@{--}@/_2pc/[dd]^{...}\ar@{--}@/_3pc/[dd]^{...}\ar@{--}[dd]\ar@{--}@/^3pc/[dd]_{...}\ar@{--}@/^2pc/[dd]_{...}\ar@{--}@/^1pc/[dd]_{...} & ... & \mu_2\ar@{-}[lldddd] & \mu_1\ar@{-}[llldddd] &\mu_0\ar@{-}[lllldddd] \\ \\ &&&& *+[F]{\{\lambda_{\mathbb{P}^1_\Kbb\setminus \{x\}}| x\in \mathbb{P}^1_\mathbb{K}\}}\ar@<-2ex>@{-}[d]^{\ ...}\ar@<-1.6ex>@{-}[d]\ar@<-1.2ex>@{-}[d]\ar@<1.2ex>@{-}[d]\ar@<1.6ex>@{-}[d]\ar@<2ex>@{-}[d]\\ &&&& *+[F]{\lambda_{\mathbb{P}^1_\Kbb}}\ar@{-}[d]\\ &&&&0}$$

\smallskip

Here the $\lambda_i$ are the homological ring epimorphisms corresponding to  preprojective silting modules, the $\mu_i$  correspond to preinjective silting modules, and the ring epimorphisms  in frames  are those with infinite dimensional codomain, that is, those given by universal localisation at (projective resolutions of) simple regular modules. The interval between $Id$ and $\lambda_{\mathbb{P}^1_\Kbb}$ represents the dual poset of subsets of the projective line $\mathbb{P}^1_\Kbb$ over $\mathbb K$. 

Up to equivalence, there is just one additional silting module which is not minimal and thus does not appear in the lattice above. It is called Lukas tilting module and it generates the class of all modules without preprojective summands.
 More details are given in \cite[Examples 5.10 and 5.18]{AMV2}.
\end{examples}

\medskip

We have seen in Section 5 that over a hereditary ring the assignment $\alpha:T\mapsto f$ from Proposition~\ref{inj1}   defines a bijective correspondence between minimal silting modules and universal localisations.
Another important  case where $\alpha$ plays a similar role  is established by Marks and Stovicek in \cite{MS1}. 

\begin{proposition}\cite{MS1}\label{inj}
If $A$ is a finite dimensional algebra, there is a commutative diagram of injections
$$\xymatrix{{\left\{\begin{array}{c}\text{\small equivalence classes of} \\ \text{\small finite dimensional}\\ \text{\small silting $A$-modules} \end{array}\right\}}\ar[rr]^{\alpha}\ar[dr]^{\mathfrak a} &  & {\left\{\begin{array}{c}\text{\small epiclasses of}\\ \text{\small universal localisations} \\ \text{\small  $f:A\to B$ with dim\,$B<\infty$} \end{array}\right\}}\ar[dl]^{\epsilon}\\ & {\left\{\begin{array}{c}\text{\small  functorially finite wide  } \\ \text{\small subcategories of $\modA$}
\end{array}\right\}} &}$$
where $\alpha$ assigns to  a silting module $T$ the associated silting ring epimorphism,  $\epsilon$ is the assignment $f\mapsto \Xcal_B\cap\modA$, 
and $\mathfrak a$  is  the map defined in (\ref{alpha}) for the abelian category  $\modA$.
\end{proposition}

\begin{proof} If $T$ is a finite dimensional silting module, then it is minimal by Example~\ref{arising}(2), and it is associated to a partial silting module $T_1$ with  projective presentation $\sigma_1\in\Mor{\projA}$. By \cite[Lemma 3.8]{MS1} the functorially finite torsion class $\gen T$ 
  gives rise to  the functorially finite wide subcategory $\mathfrak a(\gen T)=\gen T\cap T_1\,^o=\Xcal_{\sigma_1}\cap\modA$, which is precisely the wide subcategory corresponding  to 
the silting ring epimorphism $f$ associated to $T_1$. Notice that $f$ is the universal localisation of $A$ at $\sigma_1$, and it has a finite dimensional codomain by Theorem~\ref{epicl}(3). Hence $\alpha$ is well-defined, and  the diagram commutes. Finally, $\epsilon$ is obviously injective, and
 $\mathfrak a$  is injective by \cite[Proposition 3.9]{MS1}.
\end{proof}

The image of  the map $\mathfrak a$ in the proposition above is determined in \cite[Theorem 3.10]{MS1}.  It consists of the functorially finite wide subcategories  $\Wcal\subset\modA$ for which the smallest torsion class in $\modA$ containing $\Wcal$ is functorially finite. In particular, $\mathfrak a$ is  bijective whenever all torsion classes are functorially finite. Let us turn to algebras with such property.

\begin{definition}\cite{AIR,DIJ} Let $A$ be  a finite dimensional algebra, and denote by $\tau$ the Auslander-Reiten transpose. A module   $M\in\modA$ is $\tau$-{\it rigid} if  $\Hom_A(M,\tau M)=0$, and it is $\tau$-{\it tilting} if, in addition,  the number of non-isomorphic indecomposable direct summands of $M$ coincides with the number of isomorphism classes of simple $A$-modules.
The algebra  $A$ is $\tau$-{\it tilting finite} if there are only finitely many isomorphism classes of  basic $\tau$-tilting $A$-modules.\end{definition}

We list some characterisations of $\tau$-{tilting finite} algebras from \cite{DIJ}, and we add two  new conditions from \cite{AMV4}.

\begin{theorem} The following statements are equivalent for a finite dimensional algebra $A$.
\begin{enumerate}
\item $A$ is $\tau$-{tilting finite}.
\item There are only finitely many finite dimensional silting $A$-modules up to equivalence.
\item Every torsion class in $\modA$ is functorially finite.
\item Every silting $A$-module is finite dimensional up to equivalence.
\item There are only finitely many epiclasses of ring epimorphisms  $A\to B$ with $\Tor_1^A(B,B)=0$.
\end{enumerate}
In particular, if $A$ is $\tau$-{tilting finite}, all ring epimorphisms $A\to B$ with $\Tor_1^A(B,B)=0$ are universal localisations with a finite dimensional codomain.
\end{theorem}

\smallskip

\begin{theorem}\cite{MS1}\label{finite}
Let  $A$ be a finite dimensional $\tau$-{tilting finite} algebra. Then   
there are bijections between 
\begin{enumerate}
\item[(i)]  isomorphism classes of basic support $\tau$-tilting (i.e.~silting) modules;
\item[(ii)] epiclasses of ring epimorphisms  $A\to B$ with $\Tor_1^A(B,B)=0$;
\item[(ii')] epiclasses of universal localisations of $A$;
\item[(iii)] wide subcategories of $\modA$.
\end{enumerate}
\end{theorem}
The theorem above applies in particular to all algebras of finite representation type, but also to many representation-infinite  algebras.

\begin{examples}\label{exs3}
(1) \cite[Lemma 3.5]{Ma} If $A$ is a local  finite dimensional algebra, then $A$ and 0 are the  only  silting modules, up to equivalence.

\smallskip

(2) Let $A$ be a preprojective algebra of Dynkin type.
Then  the collections in Theorem~\ref{finite} are further in bijection with the elements of the Weyl group of the underlying Dynkin quiver \cite{Mi}. The  interplay between    combinatorial  and ring theoretic data is used in \cite{Ma2} to describe the algebras arising as universal localisations of $A$ and to determine the homological ring epimorphisms.

\smallskip

(3) The combinatorics of universal localisations is also used in \cite{Ma} to classify silting modules over Nakayama algebras.

\smallskip

(4) Finally, let us remark that the maps $\alpha$ and $\mathfrak a$ in Proposition~\ref{inj} do not preserve the poset structure, see e.g.~\cite[Example 4.8]{Ma} or \cite[Example 5.18]{AMV2}.

\end{examples}

We close this survey with some open questions.

\smallskip

{\bf Question 1.} What is the image of the map $\mathfrak \alpha$ in Proposition~\ref{inj}? Recall that all universal localisations are silting ring epimorphisms by Theorem~\ref{silt}, but in general we do not know how finite dimensional silting modules interact with silting ring epimorphisms having a finite dimensional codomain. 

\smallskip

{\bf Question 2.} Let $T$ be a silting module such that  $A$ has a minimal  left $\Add T$-approximation $A\to T_0$. Is  the module $T_1$ in the sequence $ A\to T_0\to T_1\to 0$   always partial silting?

\smallskip

{\bf Question 3.} In Theorems~\ref{classcomm}, \ref{her}, \ref{finite}, we have seen several instances of the interplay between silting modules  and localisations. Is there a suitable notion of localisation with a  general statement encompassing all these cases?

\end{document}